\documentclass[12pt]{article}
\usepackage[a4paper,margin=0.8in]{geometry}

\usepackage[colorlinks,citecolor=magenta,linkcolor=black]{hyperref}
\pdfpagewidth=\paperwidth \pdfpageheight=\paperheight
\usepackage{amsfonts,amssymb,amsthm,amsmath,eucal,tabu,url}
\usepackage{pgf}
 \usepackage{array}
 \usepackage{tikz-cd}
 \usepackage{pstricks}
 \usepackage{pstricks-add}
 \usepackage{pgf,tikz}
 \usetikzlibrary{automata}
 \usetikzlibrary{arrows}
 \usepackage{indentfirst}
\usepackage{tabularx} 

\theoremstyle{plain}
\newtheorem{thm}{Theorem}[section]
\newtheorem{theorem}[thm]{Theorem}
\newtheorem*{theoremA}{Main Theorem}

\newtheorem{proposition}[thm]{Proposition}
\newtheorem{corollary}[thm]{Corollary}

\theoremstyle{definition}
\newtheorem{definition}[thm]{Definition}
\newtheorem{remark}[thm]{Remark}
\newtheorem{example}[thm]{Example}

\newtheorem{problem}[thm]{Problem}

\newtheorem{thevarthm}[thm]{\varthmname}

\newenvironment{varthm*}[1]{\trivlist\item[]{\bf #1.}\it}{\endtrivlist}

\renewcommand\geq{\geqslant}

\renewcommand\leq{\leqslant}

\let\tilde=\widetilde

\newcommand\be{\begin{eqnarray*}}
\newcommand\ee{\end{eqnarray*}}

\newcommand\newop[2]{\def#1{\mathop{\rm #2}\nolimits}}
\newop\log{log}
\newop\ord{ord}
\newop\Gal{Gal}
\newop\SL{SL}
\newop\Bl{Bl}
\newop\mult{mult}
\newop\mass{mass}
\newop\div{div}
\newop\codim{codim}
\newop\sing{sing}
\newop\vdim{vdim}
\newop\edim{edim}
\newop\Ass{Ass}
\newop\size{size}
\newop\reg{reg}
\newop\satdeg{satdeg}
\newop\supp{supp}
\newop\Neg{Neg}
\newop\Nef{Nef}
\newop\Nefh{Nef_H}
\newop\Eff{Eff}
\newop\Zar{Zar}
\newop\MB{MB}
\newop\MBxC{MB\mathit{(x,C)}}
\newop\NnB{NnB}
\newop\Bigg{Big}
\newop\Effbar{\overline{\Eff}}

\def\keywordname{{\bfseries Keywords}}%
\def\keywords#1{\par\addvspace\medskipamount{\rightskip=0pt plus1cm
\def\and{\ifhmode\unskip\nobreak\fi\ $\cdot$
}\noindent\keywordname\enspace\ignorespaces#1\par}}
\def\subclassname{{\bfseries Mathematics Subject Classification
(2020)}\enspace}
\def\subclass#1{\par\addvspace\medskipamount{\rightskip=0pt plus1cm
\def\and{\ifhmode\unskip\nobreak\fi\ $\cdot$
}\noindent\subclassname\ignorespaces#1\par}}

\begin{document}
\title{On nearly free arrangements of lines with nodes and triple points}
\author{Jakub Kabat}
\date{\today}
\maketitle
\thispagestyle{empty}
\begin{abstract}
We provide a classification result on nearly free arrangements of lines in the complex projective plane with nodes and triple points.
\subclass{nearly free curves, line arrangements}
\end{abstract}

\section{Introduction}
In the recent years there is a great interest on line arrangements in the complex projective plane that are free. This is due to the reason that the celebrated Terao's conjecture predicts that the freeness of complex line arrangements in the projective plane is determined by the combinatorics. In the context of a potential counterexample to Terao's conjecture, Dimca and Sticlaru in \cite{DS} defined a new class of line arrangements that is called nearly free. Roughly speaking, if Terao's conjecture fails, then it is expected that we have two combinatorially equivalent line arrangements, one is free and the second is nearly free. Here our aim is to understand nearly free complex line arrangement that are nearly free, and they have only nodes and triple intersection points. This setting is motivated by a very recent paper by Dimca and Pokora \cite{DimcaPokora} where the authors study free conic-line arrangement with nodes, tacnodes, and ordinary triple points. Based on their ideas, we perform our classification procedure and we obtain the following result.
\begin{theoremA}
Let $\mathcal{L}\subset \mathbb{P}^{2}_{\mathbb{C}}$ be an arrangement of $d$ lines with only nodes and triple intersection points. Suppose that $\mathcal{L}$ is nearly free, then $d \in \{4,5,6,7,8\}$.
\end{theoremA}
Here is the structure of the paper. In Section 2, we recall basics on nearly free reduced plane curves in the plane. In Section 3, we provide some combinatorial constraints on the number of lines of nearly free line arrangements with nodes and triple points. In Section 4, using deformation arguments applied on free line arrangements with nodes and triple points, we perform a classification procedure that will eventually lead to Main Theorem.
\section{Introduction to nearly free curves}
 Let $C$ be a reduced curve $\mathbb{P}^{2}_{\mathbb{C}}$ of degree $d$ given by $f \in S :=\mathbb{C}[x,y,z]$. We denote by $J_{f}$ the Jacobian ideal generated by the partials derivatives $\partial_{x}f, \, \partial_{y}f, \, \partial_{z}f$. Moreover, we denote by $r:={\rm mdr}(f)$ the minimal degree of a relation among the partial derivatives, i.e., the minimal degree $r$ of a triple $(a,b,c) \in S_{r}^{3}$ such that 
$$a\cdot \partial_{x} f + b\cdot \partial_{y}f + c\cdot \partial_{z}f = 0.$$
We denote by $\mathfrak{m} = \langle x,y,z \rangle$ the irrelevant ideal. Consider the graded $S$-module $N(f) = I_{f} / J_{f}$, where $I_{f}$ is the saturation of $J_{f}$ with respect to $\mathfrak{m}=\langle x,y,z\rangle$.
\begin{definition}
We say that a reduced plane curve $C$ is \emph{nearly free} if $N(f) \neq 0$ and for every $k$ one has ${\rm dim} \, N(f)_{k} \leq 1$.
\end{definition}
To complete the picture, we need to define free curves.
\begin{definition}
We say that a reduced plane curve $C$ is \emph{free} if $N(f) = 0$.
\end{definition}
In order to study the nearly freeness of a reduced plane curve $C \, : \, f=0$, $f \in S_{d}$, Dimca and Sticlaru provided a homological criterion on the Milnor algebra $M(f) := S/J_{f}$.
\begin{theorem}[Dimca-Sticlaru]
\label{DimSti}
If $C$ is a nearly free curve of degree $d$ given by $f \in S$, then the minimal free resolution of the Milnor algebra $M(f)$ has the following form:
\begin{equation*}
\begin{split}
0 \rightarrow S(-b-2(d-1))\rightarrow S(-d_{1}-(d-1))\oplus S(-d_{2}-(d-1)) \oplus S(-d_{3}-(d-1)) \\ \rightarrow S^{3}(-d+1)\rightarrow S \rightarrow M(f) \rightarrow 0
\end{split}
\end{equation*} for some integers $d_{1},d_{2},d_{3}, b$ such that $d_{1} + d_{2} = d$, $d_{2} = d_{3}$, and $b=d_{2}-d+2$. In that case, the pair $(d_{1},d_{2})$ is called the set of exponents of $C$.
\end{theorem}
\begin{example}
Let us consider the rational cuspidal curve $C \, : \, y^2z = x^3$. It is known that $C$ is a nearly free curve. We can compute the minimal free resolution of $M(f)$ which has the following form:
$$0 \rightarrow S(-5) \rightarrow S(-4)^{2} \oplus S(-3) \rightarrow S(-2)^{3} \rightarrow S \rightarrow M(f) \rightarrow 0.$$
It means that the exponents are $(1,2)$.
\end{example}

From now on we stick to line arrangements in $\mathbb{P}^{2}_{\mathbb{C}}$. In order to study their nearly freeness, we will use \cite[Theorem 1.3]{Dimca1} which turns out to be a vital technical tool.
\begin{theorem}[Dimca]
\label{Dim}
Let $\mathcal{L}\subset \mathbb{P}^{2}_{\mathbb{C}}$ be an arrangement of $d$ lines and let $f=0$ be its defining equation. Denote by $r: = {\rm mdr}(f)$. Assume that $r\leq d/2$, then $\mathcal{L}$ is nearly free if and only if
\begin{equation}
\label{Milnor}
r^2 - r(d-1) + (d-1)^2 = \mu(\mathcal{L})+1,
\end{equation}
where $\mu (\mathcal{L})$ is the total Milnor number of $\mathcal{L}$, i.e.,
$$\mu(\mathcal{L}) = \sum_{p \in {\rm Sing}(\mathcal{L})} ({\rm mult}_{p}-1)^{2}.$$
\end{theorem}
\begin{example}
\label{d=8}
Let us consider the line arrangement $\mathcal{A} \subset \mathbb{P}^{2}_{\mathbb{C}}$ defined by the following equation
$$Q(x,y,z) = (x^2 + xy + y^{2})(y^{3}-z^{3})(z^{3}-x^{3}).$$
This arrangement consists of $8$ lines and delivers $4$ nodes and $8$ triple intersection points. Using \verb{Singular{ \cite{Singular} we can compute ${\rm mdr}(Q)$ which is equal to $4$. Since $r \leq d/2 = 4$, we can use the above criterion, namely
$$37 = 4^2 - 4\cdot 7 + 7^2 = \mu(\mathcal{A}) + 1 = 4 + 4\cdot 8 + 1,$$
so $\mathcal{A}$ is nearly free.
\end{example}

\section{Nearly free arrangements of line arrangements with nodes and triple points}
In order to provide a lower bound on the number of lines of nearly free line arrangements with nodes and triple points, we recall the following result by Dimca and Pokora \cite[Proposition 4.7]{DimcaPokora} which is adjusted to our purposes.
\begin{proposition}
Let $C \, : \, f = 0$ be an arrangement of $d$ lines in $\mathbb{P}^{2}_{\mathbb{C}}$ such that it has only nodes and triple intersection points. Then one has
$${\rm mdr}(f) \geq \frac{2}{3}d - 2.$$
\end{proposition}
If $C \, : f=0$ is a nearly free arrangement of $d$ lines with nodes and triple intersection points with the exponents $(d_{1}, d_{2})$, $d_{1}\leq d_{2}$, then ${\rm mdr}(f) = d_{1}$, and since 
$$2d_{1} \leq d_{1}+d_{2} =d$$
we obtain that ${\rm mdr}(f) \leq d/2$. Combining it with the above proposition, we arrive at
$$\frac{2}{3}d - 2 \leq {\rm mdr}(f) \leq d/2.$$
It gives us the following result.
\begin{proposition}
If $\mathcal{A} \subset \mathbb{P}^{2}_{\mathbb{C}}$ is a nearly free arrangement of $d$ lines with nodes and triple intersection points, then $d\leq 12$. 
\end{proposition}

Based on the above proposition, our goal is the following.

\begin{problem}
Classify all weak combinatorics of line arrangements with nodes and triple points in $\mathbb{P}^{2}_{\mathbb{C}}$ which are nearly free.
\end{problem}
Here by the weak combinatorics, for an arrangement of lines $\mathcal{L}$, we mean the vector $(d;t_{2},t_{3})$, where $d$ is the number of lines, $t_{2}$ is the number of nodes, and $t_{3}$ is the number of triple points. Sometimes we will write $t_{i}(\mathcal{L})$ with $i \in \{2,3\}$ in order to emphasize the underlying arrangement of lines $\mathcal{L}$.
The first step towards the classification is the following proposition.
\begin{proposition}
\label{triple}
Let $\mathcal{L} \subset \mathbb{P}^{2}_{\mathbb{C}}$ be a nearly free arrangement of $d$ lines with nodes and triple intersection points. Then
\begin{equation}
\label{t3}
t_{3} \geq \frac{1}{4}\bigg(d^{2}-4d-1\bigg).
\end{equation}
\end{proposition}
\begin{proof}
If $\mathcal{L}$ is nearly free with $r = {\rm mdr}(f)$, where $f \in \mathbb{C}_{d}[x,y,z]$ is the defining equation, then by Theorem \ref{Dim} one has
$$r^{2} - r(d-1) + (d-1)^2 = \mu(\mathcal{L})+1.$$
Let us recall that we have the following combinatorial count
$$\binom{d}{2} = \frac{d(d-1)}{2} = t_{2} + 3t_{3}.$$
Since 
$\mu(C) = t_{2} + 4t_{3} = \binom{d}{2}+t_{3}$,
we obtain
$$r^{2} - r(d-1) + (d-1)^2 = r^{2}-r(d-1)+d^{2}-2d+1 = \binom{d}{2}+t_{3}+1.$$
After simple manipulations, we arrive at
$$r^{2} - r(d-1) + \frac{d^{2}-3d-2t_{3}}{2}=0.$$
The above equation can have integer roots if $\triangle_{r} = (d-1)^2 - 2d^{2}+6d + 4t_{3} \geq 0$.
This leads us to
$$t_{3} \geq \frac{1}{4}\bigg(d^{2}-4d-1\bigg),$$
which complete the proof.
\end{proof}
In the next step, let us recall the following bound on the number of triple points for line arrangements which is due to Sch\"onheim \cite{Sch}. Define
$$U_{3}(d) := \bigg\lfloor \bigg\lfloor \frac{d-1}{2} \bigg\rfloor \cdot \frac{d}{3} \bigg\rfloor - \varepsilon(d),$$
where $\varepsilon(d)=1$ if $d = 5 \, {\rm mod}(6)$ and $\varepsilon(d) =0$ otherwise. Then
\begin{equation}
    t_{3} \leq U_{3}(d).
\end{equation}
If $\mathcal{L}$ is a nearly free arrangement with only nodes and triple intersection points, then
\begin{equation}
\label{in1}
\frac{1}{4}\bigg(d^{2}-4d-1\bigg) \leq t_{3} \leq U_{3}(d).
\end{equation}

Observe that for $d \in \{12,11,10\}$ the chain of inequalities in (\ref{in1}) leads us to a contradiction. 
\begin{corollary}
Let $\mathcal{L} \subset \mathbb{P}^{2}_{\mathbb{C}}$ be a nearly free arrangement of $d$ lines with only nodes and triple intersection points. Then
$$d \leq 9.$$
\end{corollary}
In the next chapter, we are going to indicate those values of $d \in \{4,5,6,7,8,9\}$ for which there exists a line arrangement with nodes and triple points that is nearly free. In order to do so, we are going to use some tricks regarding deletion and deformation-type arguments.
\section{Classification}
One of the tools that will help us is the following deformation type result which has a general meaning, and that is the reason why we formulate it for all reduced plane curves. Before we formulate it, let us present the following definition.
\begin{definition}
Let $\mathcal{C} \, : f=0$ be a reduced curve in $\mathbb{P}^{2}_{\mathbb{C}}$ of degree $d$. For $r = {\rm mdr}(f)$ we define the number
$$\eta(C) = r^{2} - r(d-1) + (d-1)^{2}.$$
Moreover, we define the total Tjurina number of $C$ by
$$\tau(C) := \sum_{p \in {\rm Sing}(C)} \tau_{p},$$
where $\tau_{p}$ denotes the (local) Tjurina number of $C$ at the singular point $p$.
\end{definition}
\begin{remark}
Observe that if $C$ is an arrangement of $d$ lines in $\mathbb{P}^{2}_{\mathbb{C}}$, then
$$\mu(C) = \tau(C),$$
and it follows from the fact that all the singular points for line arrangements are \emph{quasi-homogeneous}.
\end{remark}
\begin{proposition}
\label{prop:defor}
    Let $C : f = 0$ be a free reduced plane curve of degree $d$ having only nodes and $n_{3}\geq 1$ ordinary triple points. Let $C'$ be a reduced curve obtained by the following deformation performed on $C$:
\begin{center}
$(\star)$: a triple intersection point $p$ is deformed into three nodes.    
\end{center}
Denote by $f'=0$ the defining equation of $C'$. Assume that $\eta(C) = \eta(C')$ and ${\rm mdr}(f') \leq d/2$, then $C'$ is nearly free.
\end{proposition}
\begin{proof}
If $C$ is free, then by \cite[Corollary 1.2]{Dimca1} we have
$$\eta(C) = r^{2} - r(d-1) + (d-1)^{2} = \tau(C) = n_{2} + 4n_{3}.$$
Observe that 
$$\tau(C') = (n_{2} + 3) + 4(n_{3}-1) = n_{2} + 4n_{3} - 1,$$
so $\tau(C') + 1 = \tau(C)$. Since ${\rm mdr}(f')\leq d/2$, then in the light of \cite[Theorem 1.3]{Dimca1} $C'$ is nearly free, and it finishes the proof. 
\end{proof}
Let us explain this idea in detail by the forthcoming example.
\begin{example}
\label{A6}
Consider $A_{1}(6)$ arrangement defined by the equation
$$f(x,y,z) = xyz(x-y)\cdot (y-z)\cdot (x-z).$$
It is well-known that $t_{2}(\mathcal{A}_{1}(6)) = 3$, $t_{3}(\mathcal{A}_{1}(6)) = 4$, ${\rm mdr}(f) = 2$, and $\eta(\mathcal{A}_{1}(6))=19$. 
Now we consider the following deformation of $\mathcal{A}_{1}(6)$, denoted here by $\mathcal{A}_{6}$, given by the following equation
$$f'(x,y,z) = xyz(y-z)\cdot (x-z) \cdot \bigg(x - \frac{1}{2}y\bigg).$$
It has $t_{2}(\mathcal{A}_{6}) = 6$ and $t_{3}(\mathcal{A}_{6})=3$, ${\rm mdr}(f')=3$, and $\eta(A_{6})=19$.

Observe that $\mathcal{A}_{6}$ is nearly free, namely
$$\eta(\mathcal{A}_{6}) = 19 = t_{2}(\mathcal{A}_{6}) + 4t_{3}(\mathcal{A}_{6}) + 1 = \mu(\mathcal{A}_{6})+1 = \mu(\mathcal{A}_{1}(6)),$$
which completes our justification.
\end{example}
In this way, we have constructed a nearly free arrangement of $d=6$ lines. Now we are going to use results from \cite{Dums} in order to construct examples of nearly free arrangement of lines using our deformation argument, and at the end the deletion procedure.
\begin{enumerate}
    \item[$d=4$:] It is known that the maximal number of triple points for $4$ lines is equal to $1$. Moreover, this arrangement consisting of $t_{3}=1$, $t_{2}=3$, and $d=4$ is free since it is a supersolvable line arrangement. In this situation, as it turns out, we can apply our deformation argument above and conclude that an arrangement of $d=4$ lines and $6$ nodes is nearly free. 
    \item[$d=5$:] It is also known that the maximal number of triple points for $5$ lines is equal to $2$. Moreover, an arrangement with $d=5$, $t_{3}=2$, and $t_{2}$ is supersolvable, and thus it is free arrangement. Again, we can perform our deformation argument at one triple point, we obtain in that way an arrangement with $d=5$, $t_{3}=1$, and $t_{2}=7$, and this arrangement is nearly free.
    \item[$d=6$:] This case is covered by Example \ref{A6}.
    \item[$d=7$:] We know that over the complex numbers the maximal number of triple points for $7$ lines is equal to $6$. Consider the arrangement
    $$Q(x,y,z) = z\cdot (x^{2} -z^{2})\cdot (y^{2} - z^{2}) \cdot (y^{2}-x^{2}).$$
    We can check, using \verb{Singular{, that the arrangement $\mathcal{A}$ defined by $Q$ is free with the exponent $(3,3)$. Now we can deform arrangement $\mathcal{A}$ at one of the triple intersection points.
    Consider the following deformation
    $$G(x,y,z) =z\cdot (y^{2} - z^{2}) \cdot (x^{2}-z^{2})\cdot(y+x)\cdot(2y-x+2z).$$
    The arrangement $\mathcal{A}'$ defined by $G$ delivers $t_{3}=5$ and $t_{2}=6$. Using Proposition \ref{prop:defor}, we can conclude that indeed $\mathcal{A}'$ is nearly free. We can check this also directly, since ${\rm mdr}(G) = 3$, we have
    $$r^{2}-6\cdot r + 36 = 9 - 18 + 36 = \mu(\mathcal{A}') + 1 = 6 + 4\cdot 5 + 1 = 27.$$ 

\item[$d=8$:] This case is covered by Example \ref{d=8}. However, we want to explain a bit how this arrangements can be derived. Recall that the dual Hesse arrangement $\mathcal{H}$ of $d=9$ lines and $t_{3}=12$ is given by the following defining equation.
$$Q(x,y,z) = (x^3 - y^3 )\cdot (y^3 - z^3 )\cdot(z^3 - x^3 ).$$
Now we are going to remove one line. Since the picture is completely symmetric, we remove a line $\ell$ given by $x-y=0$.
Then
$$\tilde{Q}(x,y,z) = Q(x,y,z) / (x-y) = (x^{2}+xy+y^{2})\cdot(y^3 - z^3)\cdot(z^3-x^3),$$
so we arrive at the situation of Example \ref{d=8}. The procedure above is called in the literature as the deletion procedure. It is well-known that the above arrangement is nothing else than Mac Lane arrangement of $8$ lines, and it realizes the maximal number of triple points among line arrangements with $8$ lines, according to \cite{Dums}.
\end{enumerate}

The last case boils down to decide whether $d=9$ can occur, and it has a different flavour comparing with the previous cases. Using our bound (\ref{t3}) we see that for $d=9$ our nearly free arrangement should have $t_{3}\geq 11$ triple points. If $t_{3}=12$, then this is the dual Hesse arrangement which is unique up to a projective equivalence, and the dual Hesse arrangement is free. Due to this reason, our problem reduces to decide whether there exists and arrangement of $d=9$ lines with $t_{3}=11$ and $t_{2}=3$. In order to approach this problem, we use database described in \cite{Mat}. It turns out that there exists the only one matorid of $d=9$ lines and $11$ triple points. However, this matroid cannot be realized as a line arrangement over any field due to method explained in \cite[Section 4]{Bar} - in fact in this case there is  failure of Dress-Wentzl's valuation criterion. This observation finishes our classification procedure. Based on that, we are ready to formulate our main result.
\begin{theorem}
Let $\mathcal{L}\subset \mathbb{P}^{2}_{\mathbb{C}}$ be an arrangement of $d$ lines with nodes and triple intersection points that is nearly free. Then $d \in \{4,5,6,7,8\}$.
\end{theorem}
Now our aim is to find all weak combinatorics $\mathcal{C} = (d;t_{2},t_{3})$ with $d \in \{4,5,6,7,8\}$ such that $\mathcal{C}$ can be geometrically realized over the complex numbers as an arrangement of $d$ lines with prescribed $t_{2}$ and $t_{3}$ such that $\mathcal{C}$ is nearly free. We can use a classification result that comes from \cite{Mat} and Proposition \ref{triple}. 
\begin{enumerate}
    \item[$d=4$:] We should have $t_{3}\geq 0$, and the only possibility to a have nearly free and not free arrangement is $(4;6,0)$.
    \item[$d=5$:] We should have $t_{3} \geq 1$, and the only possibility to have a nearly free and not free arrangement is $(5;7,1)$.
    \item[$d=6$:] We should have $t_{3} \geq 3$, and the only possibility to have a nearly free and not free arrangement is $(6;6,3)$.
    \item[$d=7$:] We should have $t_{3} \geq 5$, and the only possibility to have a nearly free and not free arrangement is $(7;6,5)$.
    \item[$d=8$:] We should have $t_{3} \geq 7$ and it turns out that we have two possibilities, namely $(8;4,8)$ or $(8;7,7)$. We need to decide whether an arrangement with the weak combinatorics $\mathcal{C} = (8;7,7)$ is nearly free -- of course such an arrangement can be constructed over the real numbers. If such an arrangement would be nearly free, then by Theorem \ref{Dim} the following polynomial must have integer roots
    $$r^{2} -7r + 49 = r^{2}-r(d-1)+(d-1)^2 = \mu(\mathcal{C}) + 1 = t_{2} + 4t_{3}+1 = 7 + 4\cdot 7 + 1 = 36.$$
    Of course equation $r^{2} - 7r + 13=0$ does not have integer roots, so the weak combinatorics $(8;7,7)$ does not rise to a nearly free arrangement.
\end{enumerate}
Based on the above considerations, \textbf{there are exactly $5$ weak combinatorics $(d;t_{2},t_{3})$ leading to nearly free arrangements of $d$ lines with $t_{2}$ double and $t_{3}$ triple intersection points}.
\section*{Acknowledgments}
The author would like to thank Piotr Pokora for his guidance. We want to thank Lukas K\"uhne for his explanations regarding the non-existence of an arrangement with $d=9$ lines and $11$ triple points, and to Alexandru Dimca for helpful suggestions that allowed to improve the note. 

The author was partially supported by the National Science Center (Poland) Preludium Grant Nr \textbf{UMO 2018/31/N/ST1/02101}.

\vskip 0.5 cm

\bigskip
Jakub Kabat \\
Department of Mathematics,
Pedagogical University of Krakow,
ul. Podchorazych 2,
PL-30-084 Krak\'ow, Poland. \\
\nopagebreak
\textit{E-mail address:} \texttt{jakub.kabat@up.krakow.pl}
\bigskip
\end{document}